\documentclass[a5paper, 10pt]{article}

\usepackage{cite, amsthm,amsmath, amssymb, booktabs, lastpage,array, float}
\usepackage[hidelinks]{hyperref}

\def\Dj{\hbox{D\kern-.73em\raise.30ex\hbox{-}
\raise-.30ex\hbox{}}}
\def\dj{\hbox{d\kern-.33em\raise.80ex\hbox{-}
\raise-.80ex\hbox{\kern-.40em}}}

\newtheorem{theorem}{Theorem}[section]

\newtheorem{lemma}[theorem]{Lemma}

\newtheorem{proposition}[theorem]{Proposition}
\newtheorem{observation}[theorem]{Observation}

\usepackage{geometry}
\usepackage{setspace}

\usepackage{graphicx}
\makeatletter
\renewcommand{\maketitle}{
	\begin{center}
		\baselineskip=0.30in
		{\Large\bfseries \@title} \par
		\vspace{5mm}
		\baselineskip=0.2in
		{\large\bfseries \@author}\par
		\vspace{1mm}
		{\it \@address} \par
		{\small\tt \@email} \par
		\vspace{3mm}
	\end{center}
	\vspace{3mm}
}
\newcommand{\address}[1]{\def\@address{#1}}
\newcommand{\email}[1]{\def\@email{#1}}

\date{\today}

\makeatother

\newcommand{\acknowledgment}[1]{\vspace{5mm}\baselineskip=0.15in
	{\textbf{\textit{Acknowledgment\/}:} #1}
}

\usepackage{fancyhdr}

\usepackage[margin=1cm,%
			font=footnotesize,%
			format=hang,%
			labelsep=period,%
			labelfont=bf]{caption}

\title{On the Difference of Atom-Bond Sum-Connectivity
and Atom-Bond-Connectivity Indices}
\author{Akbar Ali$^{a,}$\footnote{Corresponding author}, Ivan Gutman$^{b}$, Izudin Red\v{z}epovi\'c$^{c}$,\\ Jaya Percival Mazorodze$^{d}$, Abeer M. Albalahi$^a$,\\ Amjad E. Hamza$^a$}

\address{$^a$Department of Mathematics, College of Science,\\ University of Ha\!'il, Ha\!'il, Saudi Arabia\\
$^b$Faculty of Science,\\ University of Kragujevac, Kragujevac, Serbia\\
$^c$Department of Natural Sciences and Mathematics,\\ State University of Novi Pazar, Novi Pazar, Serbia\\
$^d$Department of Mathematics,\\ University of Zimbabwe, Harare, Zimbabwe}

\email{akbarali.maths@gmail.com, gutman@kg.ac.rs,  iredzepovic@np.ac.rs, mazorodzejaya@gmail.com, a.albalahi@uoh.edu.sa, aboaljod2@hotmail.com}

\begin{document}

\maketitle

\begin{abstract}
The atom-bond-connectivity (ABC) index is one of the well-investigated degree-based topological indices. The atom-bond sum-connectivity (ABS) index is a modified version of the ABC index, which was introduced recently. The primary goal of the present paper is to investigate the difference between the aforementioned two indices, namely $ABS-ABC$. It is shown that the difference $ABS-ABC$ is positive for all graphs of minimum degree at least $2$ as well as for all line graphs of those graphs of order at least $5$ that are different from the path and cycle graphs. By means of computer search, the difference $ABS-ABC$ is also calculated for all trees of order at most $15$.
\end{abstract}

\onehalfspacing

\section{Introduction}

In this paper we consider finite simple graphs (i.e., graphs without directed, weighted,
and multiple edges, and without self-loops). Let $G$ be such a graph. In order to
avoid trivialities, it will be assumed that $G$ is connected. Its vertex set is $\mathbf V(G)$
and its edge set is $\mathbf E(G)$. The order and size of $G$ are $|\mathbf V(G)|=n$
and $|\mathbf E(G)|=m$, respectively. By an $n$-vertex graph, we mean a graph of order $n$. The degree $d_u = d_u(G)$ of the vertex $u\in \mathbf V(G)$
is the number of vertices adjacent to $u$. The edge connecting the vertices
$u$ and $v$ will be denoted by $uv$. A vertex with degree one is known as a pendent vertex.

For graph-theoretical terminology and notation used without being defined,
we refer the readers to the books \cite{ga,gb,gc}

In the early years of mathematical chemistry, Milan Randi\'c invented a
topological index \cite{g1} that eventually became one of the most
successfully applied graph-based molecular structure descriptors \cite{g2,g3,g4}.
It is nowadays called ``{\it connectivity index\/}'' or ``{\it Randi\'c index\/}''
and is defined as
$$
R = R(G) = \sum_{uv \in \mathbf E(G)} \frac{1}{\sqrt{d_u\,d_v}}\,.
$$

Much later, Zhou and Trinajsti\'c \cite{g5} proposed to consider the
variant of the connectivity index, in which multiplication is replaced by
summation, named ``{\it sum-connectivity index\/}'', defined as
$$
SC = SC(G) = \sum_{uv \in \mathbf E(G)} \frac{1}{\sqrt{d_u+d_v}}\,.
$$
The same authors examined the relations between $R$ and $SC$ \ \cite{g6}.

In 1998, Estrada et al. \cite{g7} conceived another modification of the
connectivity index, called ``{\it atom-bond-connectivity index\/}'',
defined as
$$
ABC = ABC(G) = \sum_{uv \in \mathbf E(G)} \sqrt{\frac{d_u+d_v-2}{d_u\,d_v}}\,.
$$
This molecular descriptor differs from the original connectivity index
by the expression $d_u+d_v-2$, which is just the degree of the edge $uv$
(= number of edges incident to $uv$).

Soon it was established that the $ABC$ index has valuable applicative
properties \cite{g8}. Its mathematical features were also much investigated,
see the recent papers \cite{Dimitrov-2021,Ghanbari-22,Husin-22}, the review \cite{g9}, and the references cited therein. Especially intriguing
is the fact that the apparently simple problem of finding the connected
$n$-vertex graph(s) with minimum $ABC$ index remained unsolved for about a decade \cite{Hosseini-22}.

Quite recently, the sum-connectivity analogue of the $ABC$ index
was put forward, defined as
$$
ABS = ABS(G) = \sum_{uv \in \mathbf E(G)} \sqrt{\frac{d_u+d_v-2}{d_u+d_v}}
$$
and named ``{\it atom-bond sum-connectivity index\/}''  \cite{g10}.
Until now, only a limited number of properties of the $ABS$ index were
determined. In \cite{g10}, the authors determined graphs having the minimum/maximum values of the $ABS$ index among all (i) general graphs (ii) (molecular) trees, with a fixed order; parallel results for the case of unicyclic graphs were obtained in the paper \cite{ABS-EJM}, where chemical applications of the $ABS$ index were also reported. (The general $ABS$ index corresponding to the general $ABC$ index \cite{Furtula-GABC,Das-GABC,Abreu-Blaya} was also proposed in \cite{ABS-EJM}; besides,  see \cite{new-01,new-02}.)  Alraqad et al. \cite{Alraqad-arXiv} addressed the problem of finding graphs attaining the minimum $ABS$ index over the class of all trees having given order or/and a fixed number of pendent vertices. Additional detail about the known mathematical properties of the $ABS$ index can be found in the recent papers \cite{new-03,new-04,new-05,new-06}.

As well known, if a graph $G$ has components $G_1$ and $G_2$,
then $ABC(G)=ABC(G_1)+ABC(G_2)$ and $ABS(G)=ABS(G_1)+ABS(G_2)$.
As a consequence of this, denoting by $P_2$ the graph of order
2 and size 1, the following holds.\\[2mm]
(a) If $G$ is any graph, and $G^+$ is a graph whose components are
$G$, an arbitrary number of isolated vertices, and an arbitrary number
of $P_2$-graphs, then $ABC(G)=ABC(G^+)$ and $ABS(G)=ABS(G^+)$.\\[2mm]
(b) if $G^{++}$ is a graph whose components are $G$, an arbitrary
number of isolated vertices, an arbitrary number of $P_2$-graphs, and
an arbitrary number of cycles of arbitrary size, then
\[ABC(G)-ABS(G) = ABC(G^{++})-ABS(G^{++}).\]
In order to avoid these trivialities, in what follows we consider only
connected graphs.
An obvious question is
how the two closely related molecular descriptors $ABC$ and $ABS$ are
related. In this paper, we provide some answers to this question. More precisely, we prove that the difference $ABS-ABC$ is positive for all graphs of minimum degree at least $2$ as well as for all line graphs of those graphs of order at least $5$ that are different from the path and cycle graphs. We also calculate the difference $ABS-ABC$ for all trees of order at most $15$ by utilizing computer software.

\vspace*{5mm}

\section{Main Results}\label{sec-2}
We start this section with a simple but notable result that if the minimum degree of a graph $G$ is at least $2$ then the ABS index of $G$ cannot be lesser than the ABC index of $G$.

\begin{proposition} \label{rem1}
Let $G$ be a connected non-trivial graph of order $n$, without pendent vertices. Then
$$
ABC(G) \leq ABS(G).
$$
Equality holds if and only if $G \cong C_n$, where $C_n$ is the $n$-vertex cycle.
\end{proposition}

\begin{proof}
For every edge $uv\in E(G)$, note that $d_u\,d_v \geq d_u+d_v$ with equality if and only if $d_u=d_v=2$ because $\min\{d_u,d_v\}\ge2$.
\end{proof}

If the order of a graph $G$ is one or two, then the equality
$ABC(G)=ABS(G)=0$ holds in a trivial manner.

\begin{proposition} \label{rem2}
Let $G$ be a connected graph possessing a vertex $x$ of degree $2$. Construct the graph $G^\star$ by
inserting a new vertex $y$ on an edge incident to $x$. Evidently, the degree of $y$ is also $2$.
Then
\begin{equation}          \label{ksi}
ABC(G)-ABS(G) = ABC(G^\star)-ABS(G^\star)\,.
\end{equation}
\end{proposition}

\begin{proof}
Bearing in mind the way in which the graph $G^\star$ was constructed, we see that
$$
ABC(G^\star) = ABC(G) + \sqrt{\frac{d_x+d_y-2}{d_x\,d_y}} = ABC(G) + \frac{1}{\sqrt{2}}
$$
and
$$
ABS(G^\star) = ABS(G) + \sqrt{\frac{d_x+d_y-2}{d_x+d_y}} = ABS(G) + \frac{1}{\sqrt{2}}\,.
$$
\end{proof}

Proposition \ref{rem2} implies that if there is a graph $G$ of order $n$, possessing a vertex
of degree 2, for which $ABC(G)-ABS(G)=\Theta$, then for any $p \geq 1$ there exist graphs
of order $n+p$ with the same $\Theta$-value.

The situation with graphs possessing pendent vertices is much less simple.
In what follows we present our results pertaining to trees.
By means of computer search we established the following.

\begin{observation} \label{obs1}
(a) All trees of order $n$\,, $3 \leq n \leq 10$, have the property $ABC>ABS$. \\
(b) The smallest tree for which $ABC<ABS$ is depicted in Fig. \ref{Fig4}.
For $n=11$, this tree is unique satisfying $ABC<ABS$. \\
(c) For $n=12,13,14$, and $15$, there exist, respectively, $6, 31, 134$, and $564$ distinct
$n$-vertex trees for which $ABC<ABS$. \\
(d) The tree depicted in Fig. \ref{Fig4} possess vertices of degree $2$. Therefore, from Proposition $\ref{rem2}$ it follows
that there exist $n$-vertex trees with property $ABS>ABC$ for any $n \geq 11$.
\end{observation}

\begin{figure}[!ht]
 \centering
  \includegraphics[width=0.25\textwidth]{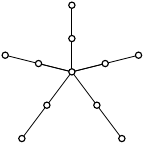}
   \caption{The smallest tree for which $ABC<ABS$.}
    \label{Fig4}
     \end{figure}

\vspace*{5mm}

\begin{observation} \label{obs2}
No tree of order $n$\,, $3 \leq n \leq 15$, has the property $ABC=ABS$. However,
there is a family of four $15$-vertex trees, shown in Fig. \ref{Fig5}, whose
$ABC$- and $ABS$-values are remarkably close. For each of these trees:
$ABC\approx10.184232$ and $ABS\approx10.184135$.
\end{observation}

\begin{figure}[!ht]
 \centering
  \includegraphics[width=0.7\textwidth]{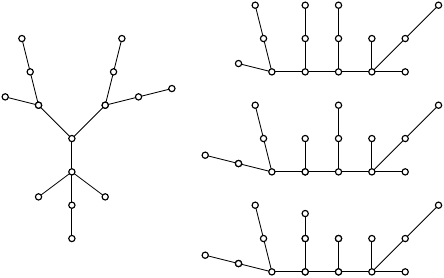}
   \caption{A family of trees with nearly equal $ABC$- and $ABS$-values.}
    \label{Fig5}
     \end{figure}

Next, we show that the inequality $ABS > ABC$ is satisfied by a reasonably large class
of graphs, namely by the line graphs.
If $G$ is the line graph of a connected $n$-vertex graph $K$ such that $2\le n \le 4$, then from the discussion made in the previous part of this section one can directly obtain the classes of graphs satisfying (i) $ABS(G)>ABC(G)$, (ii) $ABS(G)<ABC(G)$, (iii) $ABS(G)=ABC(G)$. Consequently, we assume that $n\ge 5$.

\begin{theorem} \label{thm-1}
If $G$ is the line graph of a connected $n$-vertex graph $K$ such that $n \geq 5$ and that $K \not\in \{ P_n, C_n\}$, then $ABS(G)>ABC(G)$.
\end{theorem}

In order to prove Theorem \ref{thm-1}, we need some preparations.

A decomposition of a graph $G$ is a class $\mathcal{S}_G$ of edge-disjoint subgraphs of $G$ such
that $\cup _{S\in \mathcal{S}_G} \mathbf E(S) = \mathbf E(G)$. By a clique in a graph $G$, we mean
a maximal complete subgraph of $G$. A branching vertex in a graph is a vertex of degree at least $3$. By a pendent edge of a graph, we mean an edge whose one of the end-vertices
is pendent and the other one is non-pendent. For $r\geq 2$, a path $u_1\cdots u_r$ in a graph is said to
be pendent if $\min\{d_{u_1},d_{u_r}\}=1$, $\max\{d_{u_1},d_{u_r}\} \geq 3$, and $d_{u_i}=2$ for
$2\le i \le r-1$. If $P:u_1\cdots u_r$ is a pendent path in a graph with $d_{u_r}\ge3$, we say that $P$
is attached with the vertex $u_r$. Two pendent paths of a graph are said to be adjacent if they have a common
(branching) vertex. A triangle
of a graph $G$ is said to be odd if there is a vertex of $G$ adjacent to an odd number of its vertices.

For the proof of Theorem \ref{thm-1} we need the following well-known result:

\begin{lemma}\label{lem-1}{\rm \cite{Harary}}
A graph $G$ is the line graph of a graph if and only if the star graph of order $4$
is not an induced subgraph of $G$, and if two odd triangles have a common edge then the
subgraph induced by their vertices is the complete graph of order $4$.
\end{lemma}

We can now start with the proof of Theorem \ref{thm-1}.

\begin{proof}[Proof of Theorem \ref{thm-1}]
Since $K \not \cong P_n$, the graph $G$ has at least one cycle. If $G$ is one of the two
graphs $H_1,H_2,$ depicted in Fig. \ref{Fig-1}, then one can directly verify that $ABS >ABC$ holds.
In what follows, we assume that $G\not\in \{H_1,H_2\}$.

\vspace{5mm}

\begin{figure}[!ht]
 \centering
  \includegraphics[width=0.70\textwidth]{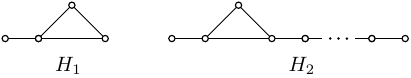}
   \caption{The graphs $H_1$ and $H_2$ mentioned in the proof of Theorem \ref{thm-1}.}
    \label{Fig-1}
     \end{figure}

Consider the difference
$$
ABS(G)-ABC(G) =\sum_{uv\in \mathbf E(G)}\left( \sqrt{\frac{d_u+d_v-2}{d_u+d_v}}-\sqrt{\frac{d_u+d_v-2}{d_u\,d_v}}\, \right)
$$
and define a function $f$ of two variables $x$ and $y$ as
$$
f(x,y) = \sqrt{\frac{x+y-2}{x+y}}-\sqrt{\frac{x+y-2}{xy}}
$$
where $y\ge x \ge 1$ and $y\ge2$. Note that the function $f$ is strictly increasing (in both $x$ and $y$).
Also, if $x$ and $y$ are integers satisfying the inequalities $y\ge x \ge 1$ and $y\ge2$, then the inequality
$f(x,y)<0$ holds if and only if $x=1$.
Thus, \[-0.129757   \approx \frac{1}{\sqrt{3}} - \frac{1}{\sqrt{2}} = f(1,2) \le f(1,y) < 0\] for every $y \ge 2$. Also, \[f(x,y) \ge f(2,3) = \sqrt{\frac{3}{5}} - \frac{1}{\sqrt{2}}\approx 0.0674899> f(2,2)=0\]
for $y\ge x \ge 2$ and $y \ge 3$. Furthermore, we have
$f(1,2)+f(2,y) >0$ for every $y\ge5$.  Thus, if either $G$ has no pendent paths or every pendent path of $G$
has length at least $2$, which is attached with a vertex of degree at least $5$, then $ABS(G)-ABC(G)>0$.
In the remaining proof, we assume that $G\not\in \{H_1,H_2\}$ and that $G$ either has at least one pendent
path of length $1$ or it has at least one pendent path of length at least $2$, which is attached with
a vertex of degree $3$ or $4$.

Let $H'$ be the graph depicted in Fig. \ref{Fig-2}, i.e., $H'$ is obtained from two disjoint graphs
$H_1$ and $H$ by identifying their vertices $z$ and $z'$.

\begin{figure}[!ht]
 \centering
  \includegraphics[width=0.50\textwidth]{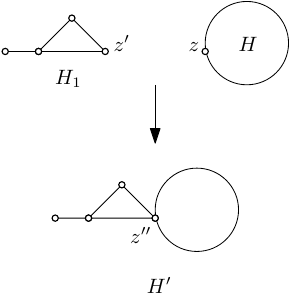}
   \caption{The graphs $H_1$, $H$, and $H'$ mentioned in the proof of Theorem \ref{thm-1}.}
    \label{Fig-2}
     \end{figure}

\vspace{3mm}

\noindent
{\bf Fact 1.} {\it If $G \cong H'$, then the sum of the contributions of the edges of $H_1$ in $G$ to the
difference $ABS(G)-ABC(G)$ is positive.}\\[2mm]

It is a well-known fact that the line graph $G$ can be decomposed into cliques, such that every edge
of $G$ lies on exactly one clique and every non-pendent vertex of $G$ lies on exactly two cliques.
Also, by Lemma \ref{lem-1}, $G$ contains no pair of adjacent pendent paths/edges and hence the number
of pendent edges/paths of $G$ is at most $\left\lfloor \,|\mathbf E(G)|/2 \right\rfloor$.
Bearing this in mind, we decompose $G$ into connected subgraphs $G_1,\ldots,G_k$ in such a way that
every $G_i$ contains at most one pendent path of $G$, such that:
\begin{description}
\item[(a)] if $G_i$ contains a pendent path of $G$ of length $1$ such that the branching vertex (in $G$)
of the considered path has a neighbor of degree $2$ in $G$, then $G_i$ is induced by the vertices of the mentioned
path and the vertices adjacent to the branching vertex (in $G$) of the mentioned path
(for an example, see $G_4$ in Fig. \ref{Fig-3}(b));
\item[(b)] if $G_i$ has a pendent path of length at least $2$ in $G$ or if $G_i$ contains a pendent path
of $G$ of length $1$ such that the branching vertex (in $G$) of the considered path has no neighbor
of degree $2$ in $G$, then $G_i$ consists of the mentioned path together with exactly one additional
edge incident with the branching vertex (in $G$) of the mentioned path (for an example, see Fig. \ref{Fig-3}).
\end{description}

\begin{figure}[!ht]
 \centering
  \includegraphics[width=0.90\textwidth]{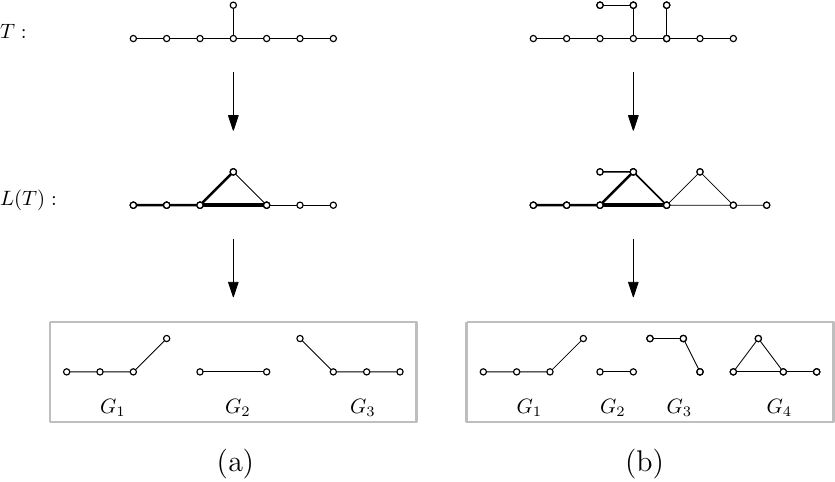}
   \caption{(a) A tree $T$, its line graph $L(T)$, and a decomposition of $L(T)$ into three connected subgraphs $G_1,G_2,G_3$.
    (b) A tree $T$, its line graph $L(T)$, and a decomposition of $L(T)$ into four connected subgraphs $G_1,G_2,G_3,G_4$.}
     \label{Fig-3}
      \end{figure}

In order to complete the proof, it is enough to show that the sum of contributions of all edges of
$G_i$ (in $G$) to the difference $ABS(G)-ABC(G)$ is positive. If a subgraph $G_i$ of $G$ contains
no pendent vertex of $G$ then certainly, the sum of contributions of all edges of $G_i$ (in $G$)
to the difference $ABS(G)-ABC(G)$ is positive.\\[2mm]

\noindent
{\bf Case 1:} a subgraph $G_i$ contains a pendent path of $G$ of length $1$, such that the
branching vertex (in $G$) of the considered path has a neighbor of degree $2$ in $G$.

Let $P:v_1v_2$ be the pendent path of $G$ contained in $G_i$, where $d_{v_1}(G)=1$ and $d_{v_2}(G)\ge3$.
Note that every neighbor of $v_2$ different from $v_1$ in $G$ has degree at least $d_{v_2}(G)-1$ in $G$ (by Lemma \ref{lem-1}).
Thus,  $d_{v_2}(G)=3$ in the case under consideration. Recall that $G\not\in \{H_1,H_2\}$ (see Fig. \ref{Fig-1}).
Consequently, $G_i \cong H_1$ and hence by Fact 1, the sum of contributions of all edges of $G_i$ to the difference
$ABS(G)-ABC(G)$ is positive.

\vspace{3mm}

\noindent
{\bf Case 2:} a subgraph $G_i$ has a pendent path of $G$ of length $1$, such that the branching
vertex (in $G$) of the considered path has no neighbor of degree $2$ in $G$.

Note that $G_i$ is a path of length $2$ in this case. Let $G_i:v_1v_2v'$, where $v_1v_2$ is a
pendent path of $G$, $d_{v_1}(G)=1$, and $d_{v_2}(G)\ge 3$.
If $d_{v_2}(G)\ge4$, then the sum of contributions of all edges of $G_i$ to the difference $ABS(G)-ABC(G)$
is positive because $d_{v'}(G)\ge d_{v_2}(G)-1$ (by Lemma \ref{lem-1}) and $f(1,y)+(y-1,y)>0$ for every $y\ge4$.
Next, assume that $d_{v_2}(G)=3$. Since $d_{v'}(G)\ge3$ in the considered case, the sum of contributions of all
edges of $G_i$ to the difference $ABS(G)-ABC(G)$ is again positive because $f(1,3) + f(3,y) > 0$ for all $y\ge 3$.

\vspace{3mm}

\noindent
{\bf Case 3:} a subgraph $G_i$ has a pendent path of length at least $2$ in $G$.

Note that $G_i$ is itself a path. Let $G_i:v_1v_2\cdots v_rv'$, where $v_1v_2\cdots v_r$ ($r\ge3$)
is a pendent path of $G$, $d_{v_1}(G)=1$, and $d_{v_r}(G)\in \{ 3,4\}$, because $G$ has no pendent
path of length at least $2$, which is attached with a vertex of degree at least $5$ (see the paragraph
appears right before the definition of $H'$ (before Fact 1)).

\vspace{3mm}

\noindent
{\bf Subcase 3.1:} $d_{v_r}(G)=3$.

The vertex $v'$ has degree at least $2$ (in $G$) and $f(1,2)+f(2,3)+f(3,y)\ge f(1,2)+f(2,3)+f(2,3)>0$ for $y\ge2$.
Thus, the sum of contributions of all edges of $G_i$ (in $G$) to the difference $ABS(G)-ABC(G)$ is positive.

\vspace{3mm}

\noindent
{\bf Subcase 3.2:} $d_{v_r}(G)=4$.

In this case, the vertex $v'$ has degree at least $3$ (in $G$) and $f(1,2)+f(2,4)+f(4,y)\geq f(1,2)+f(2,4)+f(3,4)>0$ for $y\geq 3$.
Thus, the sum of contributions of all edges of $G_i$ (in $G$) to the difference $ABS(G)-ABC(G)$ is again positive.

This completes the proof of Theorem \ref{thm-1}.
\end{proof}

\begin{theorem}\label{thm-new}
Let $G$ be a connected graph of size $m$. If the number of pendent vertices of $G$ is at most $\lfloor m/2 \rfloor$ and the number of vertices of degree $2$ in $G$ is zero, then \[ABS(G)>ABC(G).\]
\end{theorem}

\begin{proof}
Consider the function $f$ defined in the proof of Theorem \ref{thm-1}.
Here, \[-0.10939 \approx \frac{1}{\sqrt{2}} - \sqrt{\frac{2}{3}} = f(1,3) \le f(1,y) < 0\] for every $y \ge 3$. Also, \[f(x,y) \ge f(3,3) = \sqrt{\frac{2}{3}} - \frac{2}{3} \approx 0.14983\]
for $y\ge x \ge 3$.
Let $p$ denote the number of pendent vertices of $G$. Then, $m-p\ge p$. Now, by keeping in mind these observations, we have
\begin{align*}
ABS(G)-ABC(G)
&=\sum_{uv\in \mathbf E(G);\,d_u=1}\left( \sqrt{\frac{d_u+d_v-2}{d_u+d_v}}-\sqrt{\frac{d_u+d_v-2}{d_u\,d_v}}\, \right)\\[2mm]
&+ \sum_{\substack{uv\in \mathbf E(G);\\ \min\{d_u,d_v\}\ge3}}\left( \sqrt{\frac{d_u+d_v-2}{d_u+d_v}}-\sqrt{\frac{d_u+d_v-2}{d_u\,d_v}}\, \right)\\[2mm]
&\ge \sum_{\substack{uv\in \mathbf E(G);\\ d_u=1}} \left(\frac{1}{\sqrt{2}} - \sqrt{\frac{2}{3}} \right)+ \sum_{\substack{uv\in \mathbf E(G);\\ \min\{d_u,d_v\}\ge3}} \left(\sqrt{\frac{2}{3}} - \frac{2}{3}\right)\\[2mm]
&= p \left(\frac{1}{\sqrt{2}} - \sqrt{\frac{2}{3}} \right)+ (m-p) \left(\sqrt{\frac{2}{3}} - \frac{2}{3}\right)\\[2mm]
&\ge p \left(\frac{1}{\sqrt{2}} - \sqrt{\frac{2}{3}} \right)+ p \left(\sqrt{\frac{2}{3}} - \frac{2}{3}\right)\\
&>0.
\end{align*}
\end{proof}

\begin{theorem}
Let $G$ be a connected graph of size $m$ such that if $v\in V(G)$ is a vertex of degree $2$ then $v$ has no neighbor of any of the degrees $2$, $3$, $4$. If the number of pendent vertices of $G$ is at most $\lfloor m/2 \rfloor$, then $$ABS(G)>ABC(G).$$
\end{theorem}

\begin{proof}
Consider the function $f$ defined in the proof of Theorem \ref{thm-1}. Recall that
\[-0.129757   \approx \frac{1}{\sqrt{3}} - \frac{1}{\sqrt{2}} = f(1,2) \le f(1,y) < 0\] for every $y \ge 2$. Also, $f(x,y) \ge f(2,5) = \sqrt{\frac{5}{7}} - \frac{1}{\sqrt{2}} \approx 0.138047$
for $y\ge x \ge 2$ with $y\ge 5$ and
$f(x,y) \ge f(3,3)>f(2,5)$
for $y\ge x\ge3$.
Let $P$ denote the set of pendent edges of $G$. Then, $|\mathbf E(G)\setminus P | \ge |P|$. Now, by keeping in mind the above observations, we have
\begin{align*}
ABS(G)-ABC(G) &=\sum_{uv\in \mathbf E(G)\setminus P}\left( \sqrt{\frac{d_u+d_v-2}{d_u+d_v}}-\sqrt{\frac{d_u+d_v-2}{d_u\,d_v}}\, \right)\\[2mm]
&+ \sum_{uv\in P}\left( \sqrt{\frac{d_u+d_v-2}{d_u+d_v}}-\sqrt{\frac{d_u+d_v-2}{d_u\,d_v}}\, \right)\\[2mm]
&\ge \sum_{uv\in \mathbf E(G)\setminus P} \left(\sqrt{\frac{5}{7}} - \frac{1}{\sqrt{2}} \right)+ \sum_{uv\in P} \left( \frac{1}{\sqrt{3}} - \frac{1}{\sqrt{2}} \right)\\[2mm]
&= |\mathbf E(G)\setminus P| \left(\sqrt{\frac{5}{7}} - \frac{1}{\sqrt{2}} \right)+ |P| \left( \frac{1}{\sqrt{3}} - \frac{1}{\sqrt{2}} \right)\\[2mm]
&\ge | P| \left(\sqrt{\frac{5}{7}} - \frac{1}{\sqrt{2}} \right)+ |P| \left( \frac{1}{\sqrt{3}} - \frac{1}{\sqrt{2}} \right)\\
&>0.
\end{align*}
\end{proof}

\section{Conclusion and Some Open Problems}

In this paper, we considered the difference between atom-bond-connectivity ($ABC$) and atom-bond
sum-connectivity ($ABS$) indices. In the case of graphs without pendent
vertices, finding the sign of this difference is trivially easy (see Proposition \ref{rem1}). On the other hand, in the case of
graphs possessing pendent vertices, especially for trees, this difference becomes perplexed and the
complete solution of the problem awaits additional studies.

Denote the difference $ABC-ABS$ by $\Theta$. By means of computer search we found that for trees
with $n \leq 15$ vertices (except in the trivial cases $n=1,2$), $\Theta=0$ never happens. It would be of
some interest to extend this finding to higher values of $n$, or to discover a tree (or a graph with minimum degree $1$)
for which $\Theta=0$.

Let $T_n$ be the number of trees of order $n$, and $t_n$ the number of trees of order $n$ for which
$\Theta<0$. We know that $t_n/T_n > 0$ for $n \geq 11$. It is an open problem what the value of
$\lim_{n \to \infty} t_n/T_n$ is, especially whether it is equal to zero or to unity.

\acknowledgment{This research has been funded by the Scientific Research Deanship, University of Ha\!'il, Saudi Arabia, through project
number RG-23\,019.
}

\singlespacing

\end{document}